\documentclass[12pt]{article}

\usepackage{amsmath}
\usepackage{amsthm}
\usepackage{amsfonts}
\usepackage{mathrsfs}
\usepackage{stmaryrd}
\usepackage{setspace}
\usepackage{fullpage}
\usepackage{amssymb}
\usepackage{breqn}
\usepackage{enumitem}
\usepackage{bbold}
\usepackage{authblk}
\usepackage{comment}
\usepackage{hyperref}
\usepackage{pgf,tikz}
\usepackage{graphicx}
\usepackage{subcaption}

\newtheorem{thm}{Theorem}[section]

\newtheorem{lemma}[thm]{Lemma}

\newtheorem{proposition}[thm]{Proposition}
\newtheorem{prop}[thm]{Proposition}

\newtheorem{clm}[thm]{Claim}

\newcommand\ex{\ensuremath{\mathrm{ex}}}

\newcommand\cG{{\mathcal G}}
\newcommand\cH{{\mathcal H}}

\newcommand\cK{{\mathcal K}}

\newtheorem*{thm*}{Theorem}
\newtheorem*{prop*}{Proposition}
\newcommand{\ignore}[1]{}

\title{On the Turán number of the expansion of the $t$-fan}
\author{
Xin Cheng\thanks{\small School of Mathematics and Statistics, Northwestern Polytechnical University and Xi'an-Budapest Joint Research Center for Combinatorics, Xi'an 710129, Shaanxi, P.R. China. Email:
\small \texttt{xincheng@mail.nwpu.edu.cn}.}\,, \hspace{0.2em}
D\'{a}niel Gerbner\thanks{\small Alfr\'ed R\'enyi Institute of Mathematics. Email:
\small \texttt{gerbner.daniel@renyi.hu}.}\,, \hspace{0.2em} 
Hilal Hama Karim\thanks{\small Department of Computer Science and Information Theory, Faculty of Electrical Engineering and Informatics, Budapest University of Technology and Economics, Műegyetem rkp. 3., H-1111 Budapest, Hungary. E-mail: \texttt{hilal.hamakarim@edu.bme.hu}.}\,, \hspace{0.2em} 
Junpeng Zhou\thanks{\small Department of Mathematics, Shanghai University, Shanghai 200444, P.R. China. Email:
\small \texttt{junpengzhou@shu.edu.cn}.} \thanks{\small Newtouch Center for Mathematics of Shanghai University, Shanghai 200444, P.R. China.}}

\date{}

\begin{document}

\maketitle

\begin{abstract}
 The $t$-fan is the graph on $2t+1$ vertices consisting of $t$ triangles which intersect at exactly one common vertex. For a given graph $F$, the $r$-expansion $F^r$ of $F$ is the $r$-uniform hypergraph obtained from $F$ by adding $r-2$ distinct new vertices to each edge of $F$. We determine the Tur\'an number of the 3-expansion of the $t$-fan for sufficiently large $n$.
\end{abstract}



\section{Introduction}

Let $\mathcal{F}$ be an $r$-uniform hypergraph ($r$-graph for short). A hypergraph $\cH$ is \textit{$\mathcal{F}$-free} if $\cH$ does not contain $\mathcal{F}$ as a subhypergraph. The \textit{Tur\'{a}n number} of $\mathcal{F}$ is the maximum number of hyperedges in an $n$-vertex $\mathcal{F}$-free $r$-graph. We denote this maximum by ${\rm{ex}}_r(n,\mathcal{F})$. For $r=2$, we use $\mathrm{ex}(n,\mathcal{F})$ instead of $\mathrm{ex}_2(n,\mathcal{F})$, which represents a fundamental and well-studied problem in graph theory.
A classical result in extremal graph theory is Mantel's theorem \cite{Ma}, which states that the maximum number of edges in a triangle-free graph on $n$ vertices is $\lfloor\frac{n^2}{4}\rfloor$. Tur\'{a}n \cite{Tu} extended this result by determining the Tur\'{a}n number of complete graphs. 
The Erd\H{o}s-Stone-Simonovits theorem \cite{ESi,ESt} gives the asymptotic result for any $r$-chromatic graphs $F$. More precisely, $\text{ex}(n,F)=\big(\frac{r-2}{r-1}\big) \frac{n^2}{2}+o(n^2)$. 

A \textit{$t$-fan}, denoted by $F_t$, is the graph on $2t+1$ vertices consisting of $t$ triangles which intersect at exactly one common vertex. We call the common vertex the \textit{center}, the edges incident to the center \textit{close edges}, while the other edges are \textit{far edges}. Erd\H{o}s, F\"{u}redi, Gould and Gunderson \cite{EFGG} determined the Tur\'{a}n number and extremal graphs for the $t$-fan.

Hypergraph Tur\'{a}n problems are notoriously more difficult than graph versions. For example, the Tur\'{a}n number of $\cK_4^{(3)}$, the complete 3-graph on 4 vertices, is still unknown. 
For a given graph $F$, the \textit{$r$-expansion} $F^r$ of $F$ is the $r$-graph obtained from $F$ by adding $r-2$ distinct new vertices to each edge of $F$, such that the $(r-2)|E(F)|$ new vertices are distinct from each other and are not in $V(F)$. Then the copy of $F$ is a \textit{core} of this $F^r$. 
Expansions were introduced by Mubayi \cite{Mu}, see \cite{MuVe} for a survey. In particular, the Tur\'{a}n number of $K_3^r$ was first asked by Erd\H{o}s \cite{Er}, who conjectured ${\rm ex}_r(n,K_3^r)=\binom{n-1}{r-1}$ if $n\geq 3r/2$. Later, Mubayi and Verstraete \cite{MuVe2} proved the conjecture.

Motivated by the results of Erd\H{o}s, F\"{u}redi, Gould and Gunderson \cite{EFGG} and the challenges in determining hypergraph Tur\'{a}n numbers, the goal of this paper is to determine the Tur\'{a}n number for the 3-expansion of the $t$-fan. Our main result is as follows. 

\begin{thm}\label{thmnew1}
For any $t>1$, if $n$ is sufficiently large, then
  $\ex_3(n,F_t^3)=\binom{n}{3}-\binom{n-t}{3}$.  
\end{thm}

\section{Preliminaries}

Given an $r$-graph $\cH$, for a set $S$ of vertices its \textit{link hypergraph} in $\cH$ is the $(r-|S|)$-graph obtained by taking the hyperedges containing $S$ and removing $S$ from them. In the 3-uniform case, vertices have \textit{link graphs} and edges (pairs of vertices) have \textit{link sets}. 
For an $r$-graph $\cH$ and an edge $e$ inside the hyperedges of $\cH$, we say that $e$ is \textit{$q$-heavy} in $\cH$ if there are at least $q$ hyperedges of $\cH$ that contain $e$. Otherwise, we call it \textit{$(q-1)$-light}. 

Note that the link hypergraph and heavyness depends on $\cH$, and in the proof we will consider several hypergraphs (subhypergraphs of $\cH$). If we talk about a link hypergraph or a $q$-heavy/light edge $uv$ without specifying a hypergraph, we always mean the link hypergraph in $\cH$ and that $uv$ is $q$-heavy/light in $\cH$. 

A \textit{partial $r$-expansion} of a graph $F$ is obtained by taking a copy of $F$ and a subset of its edges, and adding $r-2$ distinct new vertices to those edges.

\begin{lemma}\label{greedy}
    Assume that $\cH$ contains a partial $3$-expansion of $F$ and for each edge $e$ of that copy of $F$, if $e$ was not enlarged by a new vertex, then $e$ is $(|E(F)|+|V(F)|-2)$-heavy. Then $\cH$ contains a copy of $F^3$.
\end{lemma}

Note that $(|E(F)|+|V(F)|)=|V(F^3)|$.

\begin{proof}[\bf Proof]
    We consider the copy of $F$ partially expanded and go through its edges that were not yet enlarged. For each such edge $e$, we pick a vertex from its link set that we have not used before (i.e., that is neither in the copy of $F$, nor was added to any of the edges of this copy), to create a hyperedge with $e$. Before completing the last hyperedge, we have used at most $|V(F)|+|E(F)|-1$ vertices, and two of those are in $e$, thus we can find an unused vertex that extends $e$ to a hyperedge.
\end{proof}

We will often consider copies of $S_t^3$. It is formed by $t$ hyperedges sharing a single vertex. The edges contained in those hyperedges form an $F_t$. We will use the expressions center, close edge, and far edge as in the case of $F_t$, and we will often call $S_t^3$ a star instead of the expansion of a star..
We will use a theorem of Duke and Erd\H os \cite{duer} that states $\ex_r(n,S_t^r)=\Theta(n^{r-2})$, in particular $\ex_3(n,S_t^3)=O(n)$. The exact value of $\ex_3(n,S_t^3)$ was determined in \cite{ChuFra}.

Let us assume that we are given a copy $S$ of $S_r^3$ inside a hypergraph $\cH$. There are $3r$ edges inside the hyperedges of $S$, and assume that for each such edge $e$ we are given an integer $w(e)$. We say that $S$ is \textit{nice} if for each edge $e$, 
its link set has at least $w(e)$ vertices outside $S$.

\begin{lemma}\label{nice}
    Assume that $q$ is sufficiently large with respect to $k$ and $r$, and we have a copy of $S_{q}^3$ in a hypergraph $\cH$ such that each edge $e$ inside the hyperedges of this $S_{q}^3$, is contained in at least $w(e)+1$ hyperedges, and $w(e)\le k$. Then there is a nice $S_r^3$ inside the copy of $S_{q}^3$.
\end{lemma}

\begin{proof}[\bf Proof]
    First, for each edge $e$ we fix $w(e)$ hyperedges containing $e$ that do not belong to the $S_{q}^3$. In other words, we fix $w(e)$ vertices from the link set of $e$.
    We define an auxiliary directed graph $G$. The vertices of $G$ are the hyperedges of $S_{q}^3$. For two vertices $x$ and $y$ of $G$, there is an edge from $x$ to $y$, if for some edge inside $x$ (as a hyperedge of $\cH$) we fixed a vertex that belongs to $y$, or the other way around. Then for each of the $q$ vertices of $G$, this defines at most $3k$ edges, thus there are at most $3kq$ edges in $G$. Now we forget the direction of the edges and look at the underlying undirected graph $G'$. This graph has at most $3kq$ edges. 
    By the Caro-Wei theorem \cite{caro,wei}, we have $\alpha(G')\geq\frac{n}{1+\bar{d}}$, with $\alpha(G')$ being the independence number and $\bar{d}$ being the average degree of $G'$. Therefore, there is an independent set of size $\frac{q}{1+6k}>r$ in $G'$.

An independent set of size $r$ in $G'$ corresponds to a copy of $S_r^3$ inside the copy of $S_{q}^3$. It is a nice copy, since for each edge, the $w(e)$ fixed vertices are not in this copy, otherwise there would be an edge inside $G'$ between the hyperedge that contains $e$ and the other hyperedge that contains the fixed vertex.
\end{proof}

Given a nice $S_t^3$ in a hypergraph $\cH$, we define the \textit{corresponding auxiliary bipartite graph}. Part $A$ consists of the $(5t-1)$-light edges inside the hyperedges of the nice $S_t^3$, and part $B$ consists of the vertices outside the $S_t^3$. A vertex in $A$ is joined to a vertex in $B$ if the corresponding edge and vertex together form a hyperedge of $\cH$. Observe that a matching covering $A$ corresponds to a partial copy of $F_t^3$ in $\cH$, that can be extended to a copy of $F_t^3$ by Lemma \ref{greedy}, since $|V(F_t^3)|=5t+1$ and the remaining edges inside the hyperedges of the nice $S_t^3$ are $5t$-heavy. 

We will use the following simple observation multiple times.

\begin{proposition}\label{hall}
    Let $G$ be a bipartite graph with parts $A$ and $B$ that has no matching of size $k$. If $|A|\ge k$ and each vertex of $A$ has degree at least $k-1$, then each vertex of $A$ is adjacent to the same set of $k-1$ vertices of $B$.
\end{proposition}

\begin{proof}
    There is no matching covering any $k$-set $A'\subset A$, thus by Hall's condition, there is a set $A''$ in $A$ with fewer than $|A''|$ vertices in their neighborhood. Then $A''$ is non-empty, thus there are at least $k-1$ vertices in their neighborhood, hence $|A''|\ge k$. Therefore, $A''=A'$ and each vertex of $A$ is adjacent to the same set of $k-1$ vertices of $B$. This holds for each $k$-subset of $A$, thus must hold for the whole set $A$.
\end{proof}

\section{Proof of Theorem \ref{thmnew1}}

Now we are ready to prove Theorem \ref{thmnew1}. Recall that it states that $\ex_3(n,F_t^3)=\binom{n}{3}-\binom{n-t}{3}$ for $n$ sufficiently large and $t>1$. 

\begin{proof}[\bf Proof of Theorem \ref{thmnew1}]
For the lower bound, we consider the following hypergraph. We take a set $U$ of $t$ vertices, add $n-t$ vertices and each $3$-set containing at least one vertex in $U$ as a hyperedge. 
Note that $F_t^3$ contains a matching of size $t$, i.e., $t$ independent hyperedges (corresponding to the far edges). and a vertex $u$ (the center) outside the matching such that the following holds. Any hyperedge in the matching contains two vertices $v$, $w$ (the vertices in the far edges in the core $t$-fan) such that $u,v$ and $u,w$ (the close edges) are contained in hyperedges with the third vertex not in the matching. If our construction contains $F_t^3$, then in the matching of size $t$ each hyperedge contains exactly one vertex of $U$, thus $u\not\in U$ and one of $v$ and $w$, say $v$ is not in $U$ either. But then the third vertex of the hyperedge containing $u,v$ has to be in $U$, thus in the matching, a contradiction.

For the upper bound, we consider an $n$-vertex $F_t^3$-free hypergraph $\cH$.
    Let $E_1$ denote the set of edges that are $t$-light, i.e., are contained in at most $t$ hyperedges, $E_2$ the set of $(t+1)$-heavy, $2t$-light edges and $E_3$ the set of $(2t+1)$-heavy, $3t$-light edges. Let  $\cH_i$ denote the set of hyperedges that contain at least $i$ edges of $E_i$, 
    and let $\cH_4$ denote the set of hyperedges that contain an edge of $E_2$ and two edges of $E_3$.
    Let $e_i=|E_i|$ and $h_i=|\cH_i|$. Let $\cH'$ denote the hypergraph consisting of the rest of the hyperedges. Note that 
    if the edges in a hyperedge of $\cH'$ are contained in $a,b,c$ hyperedges respectively, and $a\le b\le c$, then $a\ge t+1$, $b\ge 2t+1$ and $c\ge 3t+1$.

    \begin{clm}\label{firstcla}
        $h_1+h_2+h_3+h_4\le t(e_1+e_2+e_3)$.
    \end{clm}

    \begin{proof}[\bf Proof of Claim] 
        We count the hyperedges by taking an edge of $E_i$ first, $e_i$ ways, and extending them to hyperedges at most $it$ ways. This way we may count a hyperedge multiple times. Then for $i=1$ we obtain $h_1\le te_1$. For $i=2$, we obtain $2h_2+h_4\le 2te_2$, and for $i=3$ we obtain $3h_3+2h_4\le 3te_3$. The above inequalities imply $h_1+h_2+h_3+7h_4/6\le t(e_1+e_2+e_3)$.
    \end{proof}

    Let us return to the proof of the theorem.

    \begin{clm}\label{clami}
        $|\cH'|=O(n)$.
    \end{clm}

    \begin{proof}[\bf Proof of Claim] Let $q$ be sufficiently large.
By the theorem of Duke and Erd\H os we have that $\ex_3(n,S_q^3)=O(n)$, thus we are done if $\cH'$ is $S_q^3$-free.
    Assume that there is a copy $S$ of $S_q^3$ in $\cH'$, i.e., $uv_iw_i$ is a hyperedge of $\cH'$ for each $i\le q$. For the edges inside these hyperedges, we let $w(e)+1$ be the number of hyperedges containing $e$. Then by Lemma \ref{nice}, there is a nice copy $S'$ of $S_t^3$ in $\cH'$.

 Without loss of generality, $uv_iw_i$ is a hyperedge of $S'$ for each $i\le t$ such that for each edge $e$ inside these hyperedges, there are $w(e)$ vertices in the link graph of $e$ that do not belong to $S'$. Let $a_i\le b_i\le c_i$ denote $w(e)$ of the three edges of $uv_iw_i$. Observe that $a_i\ge t$, $b_i\ge 2t$, and $c_i\ge 3t$. 
 Let $A$ consist of an edge with $a_i$ fixed vertices from $uv_iw_i$, $B$ consist of an edge with $b_i$ fixed vertices from $uv_iw_i$ and $C$ consist of an edge with $c_i$ fixed vertices from $uv_iw_i$. 
 
 Now we go through the edges in $A$ in an arbitrary order and greedily pick a new one of those vertices for each such edge. It is doable because we always have at most
 $t-1$ vertices picked earlier, and we have at least $t$ choices. Then we go through the edges of $B$ and again greedily pick a new one of those vertices for each such edge. This is doable since we always have the $t$ vertices picked for the edges in $A$ and at most $t-1$ vertices picked for the other edges in $B$. After that, we go through the edges of $C$ and pick new vertices for them as well. This is doable, since we always have at most $3t-1$ vertices picked earlier and we have at least $3t$ choices. Thus, we can find a copy of $F_t^3$ in $\cH$, a contradiction. This completes the proof.
    \end{proof}
    Since $e_1+e_2+e_3\le \binom{n}{2}$, combining the above two claims gives the upper bound $|\cH|\le t\binom{n}{2}+O(n)$.

    Assume now that $\cH$ has $\ex_3(n,F_t^3)$ hyperedges. Then in the above calculations, each bound is sharp apart from a $O(n)$ additive term. In particular, we can assume that $h_4\le 6|\cH'|+6t^2n=O(n)$, since otherwise $|\cH|=h_1+h_2+h_3+h_4+|\cH'|=h_1+h_2+h_3+7h_4/6+|\cH'|-h_4/6\le t(e_1+e_2+e_3)+|\cH'|-h_4/6\le t\binom{n}{2}-t^2n$ and we are done. Similarly, we can obtain that $e_1+e_2+e_3=\binom{n}{2}-O(n)$. 
    Moreover, all but $O(n)$ edges of $E_i$ are in exactly $it$ hyperedges of $\cH_i$. Furthermore, there are $O(n)$ hyperedges in $\cH_2$ which contain three edges from $E_1\cup E_2\cup E_3$.

    \begin{clm}
        $|\cH_3|=O(n)$.
    \end{clm}

    \begin{proof}[\bf Proof of Claim] 
    Let $\cH_3'$ denote the set of hyperedges that contain three edges that are in exactly $3t$ hyperedges of $\cH_3$. Then $|\cH_3|=|\cH_3'|+O(n)$, since we can count the hyperedges of $\cH_3\setminus \cH_3'$ by taking one of the edges of $E_3$ in less than $3t$ hyperedges, $O(n)$ ways, and then taking the hyperedges containing that, at most $3t-1$ ways.    
    We let $w(e)=3t-1$  for the edges of $\cH_3'$. If there is a copy of $S_q^3$ for some sufficiently large $q$ in $\cH_3'$, then by Lemma \ref{nice}, we have a nice copy of $S_{2t}^3$. We claim that for each of the edges inside the hyperedges of this nice star, the same $3t-1$ outside vertices are in their link sets. 
   We consider an  $S_t^3$ subhypergraph and the corresponding auxiliary bipartite graph. Then there is no matching covering $A$ in this graph since there is no $F_t^3$ in $\cH_3'$. By Proposition \ref{hall}, the same $3t-1$ vertices are in the link set of each edge of this $S_t^3$. It holds for any $t$ hyperedges in the $S_{2t}^3$, thus it holds for all the hyperedges. Let $u$ be the center of the $S_{2t}^3$, $v$ be among these $3t-1$ vertices, and $w$ be another vertex of the $S_{2t}^3$. 
   Then $uvw\in \cH_3$, in particular $uv$ is in exactly $3t$ hyperedges. But the same holds for every vertex $w'\neq u$ of the $S_{2t}^3$, thus $uv$ is in at least $4t$ hyperedges of $\cH_3$, a contradiction. Therefore, there is no $S_q^3$ in $\cH_3'$, hence $|\cH_3'|=O(n)$ by the theorem of Duke and Erd\H os, completing the proof of this claim.
    \end{proof}

    \begin{clm}
        $|\cH_2|=O(n)$.
    \end{clm}

    \begin{proof}[\bf Proof of Claim] 
    Let $\cH_2'$ denote the set of hyperedges in $\cH_2$ with two edges in exactly $2t$ hyperedges and the third edge in at least $5t$ hyperedges. As we have observed, there are $O(n)$ edges in $E_2$ that are in less than $2t$ hyperedges, thus there are $O(n)$ hyperedges in $\cH_2$ which contain such an edge. Also, there are $O(n)$ hyperedges of $\cH_2$ that contain three edges in $E_1\cup E_2\cup E_3$. The rest of the hyperedges in $\cH_2$ contain two edges in exactly $2t$ hyperedges, and a $(3t+1)$-heavy edge. 
    There are $O(n)$ $(3t+1)$-heavy edges, and at most $5t-1$ times that many hyperedges contain a $(5t-1)$-light, $(3t+1)$-heavy edge. Thus $|\cH_2|=|\cH_2'|+O(n)$. We also have that all but $O(n)$ edges of $E_2$ are in exactly $2t$ hyperedges of $\cH_2'$.

    Consider an edge $uv$ that is in exactly $2t$ hyperedges of $\cH_2'$. In each of those hyperedges $uvw$ one of $uw$ and $vw$ is in exactly $2t$ hyperedges, and the other one is in at least $5t$ hperedges. We orient $uv$ towards $v$ if in those $2t$ hyperedges, $vw$ is $5t$-heavy for more $w$ than $uw$, and towards $u$ if $uw$ is $5t$-heavy for more $w$ than $vw$. In the case of equality we orient arbitrarily.

Let us say that a vertex $u$ is \textit{small} with respect to a hyperedge $uvw\in \cH_2'$ if $vw$ is $5t$-heavy. Then each directed edge $\overrightarrow{uv}$ has that $u$ is the small vertex with respect to at least $t$ hyperedges containing $uv$. We are going to be interested in hyperedges where at least one of the edges go out from the small vertex. If there are $O(n)$ such hyperedges, then there are $O(n)$ directed edges, thus $|E_2|=O(n)$, which implies the claim. Clearly, this is the case if for some constant $q$, for each vertex $u$, there are at most $q$ edges that go out from $u$ and are in a hyperedge where $u$ is the small vertex.

Assume that more than $q$ edges go out from $u$ and are in a hyperedge where $u$ is the small vertex. Then clearly there is an $S_r^3$ formed by such hyperedges for some $r\ge q/2t$. We let $w(e)=2t-1$ for the edges contained in exactly $2t$ hyperedges and $w(e)=5t-1$ for the $5t$-heavy edges. Then by Lemma \ref{nice}, there is a nice $S_{t+1}^3$ among these hyperedges. We consider an $S_t^3$ subhypergraph and the corresponding auxiliary bipartite graph. There is no matching covering the edges contained in exactly $2t$ hyperedges. Since each of the vertices of $A$ have degree $2t-1$, it is possible only if they have the same $2t-1$ neighbors. This means that we have vertices $u_1,\dots, u_{2t-1}$ such that each of the $2t$ edges incident to $u$ 
are extended with each $u_i$ to a hyperedge of $\cH_2'$. 
This also holds for the last hyperedge of the nice $S_{t+1}^3$. But this means that for each $i$, $uu_i$ is in at least $2t+2$ hyperedges, thus in at least $5t$ hyperedges by the definition of $\cH_2'$. Let $v\neq u$ be a vertex of the nice $S_{t+1}^3$. Then we know that $uv$ is directed towards $v$ by the choice of $S_r^3$. However, we also know that in the $2t-1$ hyperedges $uu_iv$ containing this edge $uv$, $uu_i$ is $5t$-heavy, thus we should orient $uv$ towards $u$, a contradiction, completing the proof of the claim.
    \end{proof}

    It is left to deal with the case all but $O(n)$ edges are in $E_1$ and are contained in exactly $t$ hyperedges of $\cH_1$, and all but $O(n)$ hyperedges of $\cH$ are in $\cH_1$ and contain exactly one $t$-light edge. Let $r$ be a sufficiently large integer and $p$ and $q$ be sufficiently large integers compared to $r$. Let $\cH_1'$ denote the set of hyperedges in $\cH_1$ that contain a $t$-light edge and two edges that are contained in at least $q$ hyperedges of $\cH_1$. The other hyperedges of $\cH_1$ each contain a $(t+1)$-heavy edge that is contained in at most $q-1$ hyperedges of $\cH_1$, thus $|\cH_1\setminus \cH_1'|$ is at most $q-1$ times the number of such edges, which is $O(n)$. Therefore, all but $O(n)$ hyperedges of $\cH$ are in $\cH_1'$. For each of the edges $uv$ that are in at least $q$ hyperedges of $\cH_1'$, we consider how many of those hyperedges have that the $t$-light edge is incident to $v$, and if the number is less than $p$, we delete those less than $p$ hyperedges. Similarly, if less than $p$ of those hyperedges have the $t$-light edge incident to $u$, then we delete those hyperedges. Let $\cH_1''$ denote the resulting hypergraph. Clearly, we considered $O(n)$ edges and we deleted at most $2p$ hyperedges for each such edge, thus all but $O(n)$ hyperedges of $\cH$ are in $\cH_1''$. Let $E'$ denote the set of $t$-light edges that are in exactly $t$ hyperedges of $\cH_1''$. We can count the hyperedges of $\cH_1''$ by taking an edge of $E'$ and all the $t$ hyperedges containing it, and then taking a $t$-light edge not in $E'$ and all the at most $t-1$ hyperedges of $\cH_1''$ containing it, thus $|\cH_1''|\le t|E'|+(t-1)\big(\binom{n}{2}-|E'|\big)$. Therefore, $|E'|=\binom{n}{2}-O(n)$.

    Let the weight of the $t$-light edges be $t-1$, and the weight of the other edges be $q-1$. The intersection of the link graph of a vertex $w$ in $\cH_1''$ with $E'$ is called the \textit{weak link graph} of $w$.

        \begin{clm}
        There are pairwise disjoint $t$-sets of vertices $U_1,\dots, U_k$ such that for
        each edge $uv\in E'$ 
        that is not contained in any $U_j$, the following holds. If $t>2$, then the link set of $uv$ is some $U_i$. If $t=2$, then either the link set of $uv$ is $U_i$, or $u\in U_i$, $v\in U_j$ and the link set of $uv$ is $U_i\cup U_j\setminus \{u,v\}$. 
        Furthermore, for each $U_i$ there are $r$ independent edges in $E'$ that have $U_i$ as their link set.
    \end{clm}

    \begin{proof}[\bf Proof of Claim]
         Let $uvw$ be a hyperedge of $\cH_1''$ with $uv$ being the $t$-light edge. Observe that by the definition of $\cH_1''$, in the weak link graph of $w$ every vertex has degree 0 or at least $p$. Therefore, if the link graph is not empty, it has a matching of size at least $p/2$. This corresponds to a sufficiently large star in $\cH_1''$, thus by Lemma \ref{nice} we can find a nice star $S_{r+t-2}^3$ in $\cH_1''$ with center $w$, we denote it by $S$. 
         
         Consider an $S_t^3$ inside $S$ and the corresponding auxiliary bipartite graph, then there is no matching covering $A$, thus by Proposition \ref{hall} all vertices of $A$ have the same $t-1$ neighbors. It means that there are $t-1$ vertices $u_2,\dots,u_t$ that are in the link set of the $t$ $t$-light edges in $S$. This holds for any set of $t$ hyperedges in $S$, so this must hold for each of the hyperedges. Therefore, $U:=\{w,u_2,\dots,u_t\}$ is a $t$-set such that there are at least $r$ independent edges in $E'$ that have $U$ as their link set. Note that it is possible that $U$ contains $u$ or $v$.

          If one of $u$ and $v$, say $u$ is in $U$, then consider the link set of $uv$. For each vertex $w'$ in the link set, we can apply the above procedure and obtain a $t$-set $U'$ containing $w'$ that is the link set of $r$ independent edges. If there are two vertices $u',u''\in U'\setminus U$, then we take $t-1$ independent edges with $U$ as their link set such that neither of them contains $u'$, nor $u''$, nor $v$ and we pair them with distinct vertices $z$ from $U\setminus \{u\}$. For each such edge $xy$, we take the triangle $uxy$, and for the $t$-light edge $xy$, we take the hyperedge $xyz$, where we paired $xy$ with $z\in U\setminus \{u\}$. We also take the triangle $uu'v$ and for the $t$-light edge $uv$ we take the hyperedge $uu''v$. We fixed a copy of $F_t$ and we fixed the hyperedges for the $t$-light edges. The rest of the edges are $q$-heavy, thus we can find a copy of $F_t^3$ by Lemma \ref{greedy}, a contradiction. The same argument works if $U$ contains two vertices not in $U'$. If $t>3$, then either $U'\setminus U$ or $U\setminus U'$ has size two, while if $t=2$, then both have size one if and only if the link set of $uv$ is $U\cup U'\setminus \{u,v\}$.

         For an arbitrary $t$-light edge $uv$, we found a $t$-set $U$ satisfying the desired properties. We also proved the furthermore part. It is left to show that the $t$-sets $U$ obtained this way are pairwise disjoint. Assume otherwise that $U\neq U'$ but $u\in U\cap U'$. Then we have $u_1,\dots,u_{t-1}\in U\setminus \{u\}$ and $u_t\in U'\setminus U$. Let us take $t-1$ independent edges $x_1y_1,\dots,x_{t-1}y_{t-1}$ that have $U$ as their link set and do not contain $u_t$. Let us take an edge $x_ty_t$ with link set $U'$ such that $x_t,y_t\not\in U$. Then we take the triangles $ux_iy_i$, and for the $t$-light edges we take the hyperedges $x_iy_iu_i$. Again, we fixed a copy of $F_t$ and we fixed the hyperedges for the $t$-light edges. The rest of the edges are $q$-heavy, thus we can find a copy of $F_t^3$ by Lemma \ref{greedy}, a contradiction completing the proof.
    \end{proof}

    If $t>2$, then for each edge $uv\in E'$ that is not inside any $U_j$, there is a unique set $U_i$ given by the above claim that is the link set of $uv$. We say that the color of the edge $uv$ is $i$. Then all but $O(n)$ edges have a color. In the case of $t=2$, we do not define the color of the edges if $u\in U_i$, $v\in U_j$ with $i\neq j$.

    \begin{clm}
        All but $O(n)$ edges of $E'$ have the same color.
    \end{clm}

\begin{proof}[\bf Proof of Claim]

First, we show that each vertex $v$ is incident to edges of at most one color, if $t$ is even, and to edges of at most two colors if $t$ is odd. Observe that if a vertex is incident to an edge of color $i$, then it is incident to at least $p$ edges of color $i$, since it has non-zero degree in the weak link graph of a vertex of $U_i$. Assume that $vw$ is of color $1$ and $v\not\in U_1=\{u_1,\dots,u_t\}$, then we consider the hyperedges $vu_1w_1,\dots, vu_{\lfloor t/2\rfloor} w_{\lfloor t/2\rfloor}$, where $w_i$ are distinct vertices in the link set of $vu_i$. The edges $vw_i$ are extended to hyperedges with $u_{t-i+1}$. Each other edge of these hyperedges is $q$-heavy, thus we can greedily extend them by Lemma \ref{greedy} to obtain an $F_{\lfloor t/2\rfloor}^3$. If $v\in U_1$, say $v=u_t$, we get an $F_{t-1}^3$ with center $v$ the more usual way: we consider $t-1$ independent edges $x_iy_i$ in $E'$ that have $U_1$ as their link set, extend each edge $x_iy_i$ with $u_i$, while the edges $vx_i,vy_i$ are $q$-heavy, thus we can greedily extend them by Lemma \ref{greedy} to obtain an $F_{t-1}^3$. Observe that each of the vertices of these subhypergraphs not in $U_1$ was chosen greedily, thus we can take copies of $F_{\lfloor t/2\rfloor}^3$ that share only $v$ for each color incident to $v$.

Consider now the case $t=2$. Not every edge of $E'$ has a color in this case, for example if $U_1=\{u_1,u_2\}$ and $U_2=\{u_3,u_4\}$, then $u_1u_3$ may be an edge of $E'$. In that case $u_1u_2u_3$ and $u_1u_3u_4$ are in $\cH_1''$. However, then $u_1u_4$ and $u_2u_3$ are $q$-heavy. Therefore, between $U_1$ and $U_2$, there are at least as many $q$-heavy edges as edges of $E'$. This holds for each pair $U_i,U_j$, thus, there are $O(n)$ edges of $E'$ between such sets. Therefore, for each $t$, all but $O(n)$ edges are colored edges of $E'$.
    
    If $t$ is even, each vertex is incident to edges of at most one color, thus there is no colored edge of $E'$ between the set of vertices incident to color $i$ and the set of vertices incident to color $j$. Let $A_i$ denote the set of vertices incident to some edge of color $i$, and $A_0$ denote the set of vertices not incident to any colored edges. Since there is no colored edge of $E'$ between $A_i$ and $A_j$, we have $|A_i||A_j|=O(n)$. 
    This holds for each pair $i,j$, hence it is easy to see that all but one of the sets $A_i$ contain $O(1)$ vertices, thus all but one of the colors have $O(1)$ edges, completing the proof in this case.

    If $t$ is odd, it is possible that some vertices are incident to edges of colors $i$ and $j$ both. However, we claim that there is no edge between such vertices. 
    Indeed, assume that $vw$ is such an edge, and without loss of generality, the color of $vw$ is not $i$. We have an $F_{t-1}^3$ with center $v$ by the above argument (if $v\in U_i\cup U_j$, then we have an $F^3_t$, a contradiction). There is
 a vertex $u'\in U_i$ that extends some $t$-light edge containing $v$ such that $u'$ is not used in this $F_{t-1}^3$, and analogously a vertex $u''\in U_j$. We will only use that $u'v$ and $u'w$ are $q$-heavy. We consider the triangle $vu'w$ as our last triangle. The edge $vw$ is extended by $u''$, while the other two new edges are extended greedily to obtain an $F_t^3$, a contradiction. We remark that this is the first time in the proof where we may use a triangle in the core $F_t$ that is not a hyperedge. 

 Next we show that if $v\in U_i$ and $vw\in E'$ ($w\notin U_i$), then $w$ is also incident to another color $j$. Observe that $vw$ has a vertex $u'\not\in U_i$ in its link set. Then $w$ has degree at least $p$ in the weak link graph of $u'$. In particular, an edge $ww'$ of $E'$ has $u'$ in its link set. If $w\in U_j$ for some $j(\neq i)$, then we can pick $w'\not\in U_j$. Then the color of $ww'$ cannot be $i$, and $ww'$ has a color.

Observe that any edge between $U_i$ and $U_j$ is not in $E'$ since $t>1$ is odd. Therefore, the number $k$ of distinct sets $U_i$ is $O(\sqrt{n})$, hence there are $n-O(\sqrt{n})$ vertices not in $U_i$. The edges between $\cup_{i=1}^k U_i$ and the rest of the vertices are not in $E'$, thus there are $O(n)$ such edges, which implies that $|\cup_{i=1}^k U_i|=O(1)$, i.e., $k=O(1)$.

 We claim that there are $O(\sqrt{n})$ vertices that are incident to two colors. Let $A_{i,j}$ denote the set of vertices incident to colors $i$ and $j$. Then we know that $\sum_{i,j\le k} \binom{|A_{i,j}|}{2}=O(n)$, because this counts the edges inside the sets $A_{i,j}$, which are not in $E'$. Then, by the inequality of the quadratic and arithmetic mean, $\sum_{i,j\le k}|A_{i,j}|\le O(\sqrt{n})$. 

Recall that $A_i$ denotes the set of vertices incident to edges of color $i$ and no other color. Then we partition the sets $A_i$ into two parts. No matter how we choose the partition, there are no edges of $E'$ between the parts, thus there are $O(n)$ such edges. One of the parts has $\Theta(n)$ vertices, hence the other part must have $O(1)$ vertices. We further partition the larger part as long as it consists of more than one part. We obtain that all but one of the sets $A_i$ has size $O(1)$. Such a color $i$ has that all the edges of color $i$ are inside $A_i\cup A_{i,1}\cup\dots\cup A_{i,k}$. Recall that $|A_{i,j}|=O(\sqrt{n})$. Therefore, $|A_i\cup A_{i,1}\cup\dots\cup A_{i,k}|=O(\sqrt{n})$, thus there are $O(n)$ edges of color $i$. We have $O(1)$ colors, completing the proof.
\end{proof}

We obtained that there is a set $U_1=\{u_1,\dots,u_t\}$ that is the link set of $\binom{n}{2}-O(n)$ edges. We will now return to $\cH$ and compare it to the hypergraph $\cH_0$ we obtain by taking each hyperedge that contains at least one of the vertices of $U_1$. Let $\cG$ denote the set of hyperedges in $\cH$ that do not contain any element of $U_1$. Then clearly $|\cG|=O(n)$ and at most $|\cG|$ hyperedges that intersect $U$ are missing from $\cH$, since $|E(\cH)|\geq|E(\cH_0)|$. 

Let $Q$ denote the set of vertices outside $U_1$ that have $q$-heavy edges to each vertex of $U_1$, and $R$ denote the set of other vertices outside $U_1$. 
Observe that each edge from a vertex of $U_1$ to an endpoint of an edge of $E'$ is $q$-heavy, thus each edge incident to a vertex of $R$ is not in color 1, hence $|R|=O(\sqrt{n})$.
We claim that each hyperedge of $\cG$ contains at least two vertices of $R$. Indeed, if $vwx\in \cG$ with $v,w\in Q$, then consider the triangles $u_1vw$, $u_1v_iw_i$ for $i\ge 2$, when $v_iw_i$ are arbitrary independent edges with $u$ as their link set that do not contain $v,w$ and are in $Q$. Then $vw$ is extended to a hyperedge with $x$, $v_iw_i$ are extended to a hyperedge by $u_i$, while the other edges are $q$-heavy, thus we can greedily obtain an $F_t^3$, a contradiction. Let us remark that this is the second time that we obtain an $F_t^3$ where one of the core triangles does not necessarily form a hyperedge.

Let $R'$ denote the set of vertices in $R$ that are incident to a hyperedge of $\cG$ with a vertex in $Q$. We claim that there is no $q$-heavy edge between $U_1$ and $R'$. Indeed, if $uv$ is $q$-heavy, for some $v \in R'$, and $vwx\in \cG$ with $w\in Q$, then we take the triangle $uvw$, and choose hyperedges $uv_iw_i$ for $1\leq i\leq t-1$ inside $Q$ as above. Then,  we extend $vw$ by $x$ and $v_iw_i$ by the $t-1$ other vertices in $U_1$, and the rest of the edges are $q$-heavy, so we can greedily extend them, and obtain a copy of $F_t^3$, a contradiction. 

If $\cG$ is not empty, then there is a vertex $v\in R$. Then there is $u\in U_1$ such that $uv$ is $(q-1)$-light, thus $\Omega(n)$ hyperedges containing both $u$ and $v$ are missing from $\cH$, hence $|\cG|=\Omega(n)$. More precisely, we have at least $t(n-t-q)|R'|$ and at least $(n-t-q)(|R|-|R'|)$ missing hyperedges. Then 
\begin{eqnarray}
|\cG|\geq t(n-t-q)|R'|+(n-t-q)(|R|-|R'|)=((t-1)|R'|+|R|)n-O(1).
\end{eqnarray}
Recall that $|\cG|=O(n)$, hence $|R|=O(1)$, and there are $O(1)$ hyperedges inside $R$. Therefore, there is a vertex in $R'$ and $|\cG|=\Theta(n)$. Then there are $\Theta(n)$ hyperedges in $\cG$ that contain a vertex from $Q$.

Consider the $q$-heavy edges inside $R'$. If there are at most $t|R'|-1$ such edges, then we count $|\cG|$ by picking such an edge first, then we obtain the upper bound $n(t|R'|-1)+O(1)$, which is a contradiction to inequality (1). 
Moreover, if each vertex of $Q$ that is incident to some edge of $E'$ has at most $t|R'|-1$ $q$-heavy edges in its link graph in $\cG$, then we again obtain the upper bound $n(t|R'|-1)+O(1)$, which is a contradiction to inequality (1).  

Otherwise, there is a vertex $v\in Q$ with more than $t|R'|-1$ $q$-heavy edges inside $R'$ that extend $v$ to hyperedges. By the Erd\H os-Gallai theorem \cite{EGa}, there is a $P_{2t+1}$ among such edges, moreover, there is also a $P_{2t+2}$ unless those edges form vertex-disjoint copies of $K_{2t+1}$.

We cut that path to $\lceil t/2\rceil$ vertex-disjoint copies of $P_4$ if possible. For each such path $w_1w_2w_3w_4$, we consider the triangle $vw_2w_3$. For the edge $vw_2$, we use the hyperedge $vw_1w_2$, for the edge $vw_3$ we use the hyperedge $vw_3w_4$, while the edge $w_2w_3$ is $q$-heavy. This way we obtain $\lceil t/2\rceil$ triangles. We also have $\lfloor t/2\rfloor$ vertex-disjoint triangles of the form $vu_iz_i$, where $i$ is even and $vz_i$ is an arbitrary edge of color 1. The edge $vz_i$ is extended to a hyperedge by $u_{i-1}$.
These triangles form an $F_t$, and for each edge of this $F_t$, either we fixed a hyperedge containing it and no other vertices of the $F_t$ or the other fixed hyperedges, or the edge is $q$-heavy, thus we can find an $F_t^3$, a contradiction. 

If we cannot cut the path to $\lceil t/2\rceil$ vertex-disjoint copies of $P_4$, then $t$ is odd, and the $q$-heavy edges inside $R'$ in the link graph in $\cG$ of $v$ form cliques of order $2t+1$. If there is more than one such clique, we can find $\lceil t/2\rceil$ vertex-disjoint copies of $P_4$ among those edges and proceed as above. If there is only one such clique, then $|R'|=2t+1$ (indeed, if $2t+1<|R'|<2(2t+1)$, then $|E(K_{2t+1})|+|E(K_{|R'|-2t-1})|\leq t|R'|-1$, a contradiction). Let $Q'$ denote the set of vertices in $Q$ that do not have all the edges inside $R'$ in their link graph in $\cG$. Then we have that $|\cG|\le t(2t+1)(|Q|-|Q'|)+(t(2t+1)-1)|Q'|+O(1)=t(2t+1)|Q|-|Q'|+O(1)$, thus $|Q'|=O(1)$ by inequality (1). In particular, there is an edge of $E'$ such that both endpoints are not in $Q'$, without loss of generality let $vv'$ be that edge.

Now from $R'$ we take $\lfloor t/2\rfloor$ vertex-disjoint copies of $P_4$ and a copy of $P_3$ $w_1w_2w_3$. We obtain $\lfloor t/2\rfloor$ triangles this way, and $\lfloor t/2\rfloor$ triangles of the form $vu_iz_i$. Recall that $u_t$ is not in any of these triangles, nor in the hyperedges fixed for the edges $vz_i$. Consider the triangle $vv'w_2$. Note that this is the third time we use a triangle that may not be a hyperedge. For the edge $vw_2$, we use the hyperedge $vw_2w_1$. For the edge $v'w_2$, we use the hyperedge $v'w_2w_3$. For the edge $vv'$, we use the hyperedge $vv'u_t$. Then we found $t$ triangles forming an $F_t$, and for each edge of this $F_t$, either we fixed a hyperedge containing it and no other vertices of the $F_t$ or the other fixed hyperedges, or the edge is $q$-heavy, thus we can find an $F_t^3$, a contradiction. Thus, $|\cG|=0$ and $|\cH|\leq \binom{n}{3}-\binom{n-t}{3}$, completing the proof. 
\end{proof}

\section{Concluding remarks}

Let $F_t^3(i)$ denote the hypergraph we obtain from $F_t^3$ by adding $i$ triangles of the core $F_t$ as hyperedges. The proof of Theorem \ref{thmnew1} actually shows that if $i<t$, then $\ex_3(n,F_t^3(i))=\binom{n}{3}-\binom{n-t}{3}$ for sufficiently large $n$, and $\ex_3(n,F_t^3(t))=t\binom{n}{2}+O(n)$. It seems likely that $\ex_3(n,F_t^3(t))=\binom{n}{3}-\binom{n-t}{3}$ also holds for sufficiently large $n$.

The proof of Theorem \ref{thmnew1} also gives stability in the following sense. If an $F_t^3$-free hypergraph is not a subhypergraph of the extremal construction, then it has $t\binom{n}{2}-\Omega(n)$ hyperedges. 


A \textit{$(t,k)$-fan}, denoted by $F_{t,k}$, is the graph on $(k-1)t+1$ vertices consisting of $t$ copies of $K_k$ sharing exactly one vertex. Tang, Li and Yan \cite{tly} determined $\ex_3(n,F_{t,k}^3)$ for sufficiently large $n$ if $k>3$. Note that this is a very different problem, the so-called \emph{non-degenerate} case, where the chromatic number is larger than the uniformity, thus it is obvious that $\ex_3(n,F^3_{t,k})=\Theta(n^3)$. Moreover, the asymptotics is also known for any graph in the non-degenerate case by a theorem in \cite{MuVe}, see \cite{pttw} for a proof. Gerbner \cite{Ge1} determined $\ex_r(n,F^r_{2,k})$ if $k>r$ and $n$ is sufficiently large.

A natural question is to study $\ex_r(n,F_t^r)$ or more generally $\ex_r(n,F_{t,k}^r)$ for $r>3$. Our methods do not seem to be applicable for this problem, but we can determine the order of magnitude.

\begin{proposition}
    $\ex_r(n,F_{t,k}^r)=\Theta(n^r)$ if $k>r$ and $\Theta(n^{r-1})$ if $2<k\le r$.
\end{proposition}

\begin{proof}[\bf Proof]
We have mentioned that in the non-degenerate case $k>r$ even the asymptotics is known \cite{MuVe}, thus we only have to deal with the case $k\le r$. The lower bound is given by taking all the hyperedges containing a fixed element. Indeed, this hypergraph does not contain two independent hyperedges, while $F_{t,k}^r$ does.

    Gerbner \cite{gerb} proved that for any $F$, $\ex_r(n,F^r)=O(n^{r-1})+\ex(n,K_r,F)$, where $\ex(n,K_r,F)$ is the largest number of copies of $K_r$ in $n$-vertex $F$-free graphs. Alon and Shikhelman \cite{alon} proved that $\ex(n,K_3,F_t)=O(n)$, the exact value of $\ex(n,K_3,F_t)$ was determined in \cite{zcggyh}. Gerbner \cite{Ge1} proved $\ex(n,K_r,F_t)=O(n)$ for $r>3$. Combining these completes the proof for $k=3$.

    For larger $k\le r$, we are not aware of any result on $\ex(n,K_r,F_{t,k})$, but we can prove a bound suitable for our purposes. Recall that $F_{t,k}$ consists of a vertex joined to each vertex of $t$ vertex-disjoint copies of $K_{k-1}$, which we denote by $tK_{k-1}$. In an $n$-vertex $F_{t,k}$-free graph we can count the copies of $K_r$ by picking a vertex first, and then a $K_{r-1}$ from its neighborhood. The neighborhood must be $tK_{k-1}$-free, thus $\ex(n,K_r,F_{t,k})\le n\cdot \ex(n-1,K_{r-1},tK_{k-1})$. The order of magnitude of $\ex(n-1,K_{r-1},tK_{k-1})$ was determined in \cite{gmv}, its exact value was determined in \cite{ger}. We have $\ex(n-1,K_{r-1},tK_{k-1})=\Theta(n^x)$, where $x=\big\lfloor \frac{(k-1)t-r}{t-1}\big\rfloor-1\le k-2\le r-2$. Combining the above bounds gives $\ex(n,K_r,F_{t,k})=O(n^{r-1})$, thus $\ex_r(n,F_{t,k}^r)=O(n^{r-1})$, completing the proof.
\end{proof}

One can define expansions of hypergraphs analogously to expansions of graphs. In particular, we consider the $(r-1)$-graph on $t(r-1)+1$ vertices that consists of $t$ copies of the $(r-1)$-uniform $r$-clique $\cK_r^{(r-1)}$ that intersect at exactly one common vertex. Its $r$-expansion is obtained by adding a new vertex to each of the $tr$ hyperedges such that these $tr$ new vertices are distinct from each other and are not in the original $(r-1)$-graph. Let us denote this $r$-graph by $F(t,r)$. 

\begin{thm}
    For any $t>1$ we have that $\ex_r(n,F(t,r))=t\binom{n}{r-1}+O(n^{r-2})$.
\end{thm}

The proof goes the same way as the first part of the proof of Theorem \ref{thmnew1}, hence we only present a sketch. 

\begin{proof}[\bf Sketch of proof]
For the lower bound, we consider the following $r$-graph. We take a set $U$ of $t$ vertices, add $n-t$ vertices and each $r$-set containing at least one vertex in $U$ as a hyperedge. Similar to the proof of the lower bound in Theorem \ref{thmnew1}, we can obtain that the resulting $r$-graph is $F(t,r)$-free. 

For the upper bound, we consider heavy $(r-1)$-edges and nice copies of $S_t^r$ analogously, in particular Lemma \ref{nice} extends to our setting. We define sets of $(r-1)$-edges $E_i$ for $i\le r$ as in the proof of Theorem \ref{thmnew1}. Let $\cH_i$ denote the set of $r$-edges that contain at least $i$ $(r-1)$-edges of $E_i$, and $\cH'$ denote the set of hyperedges such that their subedges of size $r-1$ have an ordering with the following property: the $i$th $(r-1)$-edge is $(it+1)$-heavy. Then, as in Claim \ref{clami}, for some $p$ and $q$, a nice $S_p^r$ could be used to find greedily an $F(t,r)$, thus $\cH'$ does not contain a nice $S_p^r$, thus does not contain an $S_q^r$, thus has size $O(n^{r-2})$ by the theorem of Erd\H os and Duke. 

Let $\cH_0$ denote the rest of the hyperedges. In the case $r=3$, we had one more type of hyperedges, those with two edges of $E_3$ and one edge of $E_2$. For larger $r$, there are more types of hyperedges in $\cH_0$. We consider a partition of $\cH_0$ into sets $\cH_{a_1,\dots,a_r}$, where the hyperedges in $\cH_{a_1,\dots,a_r}$ contain exactly $a_i$ $(r-1)$-edges from $E_i$.
Observe that there is an $i$ such that $a_1+\dots+a_i\ge i$ since otherwise $\cH_{a_1,\dots,a_r}$ would be a subhypergraph of $\cH'$. Now we apply the argument from Claim \ref{firstcla}, we count the hyperedges by taking an $(r-1)$-edge of $E_i$ first, and extending them to hyperedges at most $it$ ways. This way we may count a hyperedge multiple times. For any $i$, we obtain that $it|E_i|$ at at least $i|\cH_i|$, and for each possible sequence $a_1,\dots,a_r$, we add $a_i|\cH_{a_1,\dots,a_r}|$. Adding up these inequalities for each $i$, we have that $|\cH_{a_1,\dots,a_r}|$ appear with a coefficient of at least one, thus their sum is at least $|\cH_0|$. Therefore, we have $t\binom{n}{r-1}\ge t(|E_1|+\dots+|E_r|)\ge |\cH_1|+\dots+|\cH_r|+|\cH_0|$, completing the proof.
\end{proof}

\bigskip
\textbf{Funding}: 
The research of Cheng is supported by the National Natural Science Foundation of China (Nos. 12131013 and 12471334), Shaanxi Fundamental Science Research Project for Mathematics and Physics (No. 22JSZ009) and the China Scholarship Council (No. 202406290241). 

The research of Gerbner is supported by the National Research, Development and Innovation Office - NKFIH under the grant KKP-133819.

The research of Zhou is supported by the National Natural Science Foundation of China (Nos. 11871040, 12271337, 12371347) and the China Scholarship Council (No. 202406890088).

\end{document}